\documentclass{amsart}%
\usepackage{amsmath}
\usepackage{graphicx}
\usepackage{amsfonts}
\usepackage{amssymb}%
\providecommand{\U}[1]{\protect\rule{.1in}{.1in}}
\newtheorem{theorem}{Theorem}
\theoremstyle{plain}

\newtheorem{corollary}{Corollary}

\newtheorem{lemma}{Lemma}

\newtheorem{remark}{Remark}

\numberwithin{equation}{section}
\begin{document}
\title[Poly-analytic Functions]{On the univalence of Poly-analytic Functions}
\author{Zayid AbdulHadi}
\address{Department of Mathematics, American University of Sharjah,Sharjah, Box 26666, UAE}
\email{zahadi@aus.edu}
\author{Layan El Hajj}
\address{Department of Mathematics, American University of Dubai,Dubai, Box 28282, UAE. }
\email{lhajj@aud.edu }
\date{February 24,2020.}

\begin{abstract}
A continuous complex-valued function $F$ in a domain $D\subseteq\mathbf{C}$ is Poly-analytic of order $\alpha$ if it satisfies $\partial^{\alpha}_{\overline{z}}F=0.$ One can  show that $F$ has the form $F(z)={\displaystyle\sum\limits_{0}^{n-1}}\overline{z}^{k}A_{k}(z)$,
 where each $A_k$ is an analytic function$.$ In this paper, we prove the existence of a Landau constant for Poly-analytic functions and the special Bi-analytic case.  We also  establish the Bohr's inequality for
poly-analytic and bi-analytic functions which map $U$ into $U$. In addition,
we give an estimate for the arclength over the class of poly-analytic mappings and consider the problem of minimizing moments of order $p$.

\end{abstract}
\maketitle

\section{\bigskip\qquad Introduction}
There exist non-analytic functions with significant structure and with properties reminiscent of those satisfied by
analytic functions. Such nice non-analytic functions are called Poly-analytic functions.
 A continuous complex-valued function $F$ defined in a domain $D\subseteq
\mathbf{C}$  is Poly-analytic
of order $\alpha$ if it satisfies the generalized
Cauchy–Riemann equations $\partial^{\alpha}_{\overline{z}}F=0.$ An iteration argument shows that  $F$ has the form
\begin{equation}
F(z)=%
{\displaystyle\sum\limits_{0}^{\alpha-1}}
\overline{z}^{k}A_{k}(z),
\end{equation}
where $A_k$ are analytic functions for $k=1,...\alpha-1.$In particular,  a continuous complex-valued function $F=u+iv$ in a domain $D\subseteq
\mathbf{C}$ is said to be Bi-analytic if $\dfrac{\partial}{\partial
{\overline{z}}}\left(  \dfrac{\partial F}{\partial{\overline{z}}}\right)  =0.$ It is a poly-analytic function of order $2$.
 Note that $F_{\overline{z}}$ is analytic in $D.$
In any simply connected domain $D,$ it can be shown that $F$ has the form%
\begin{equation}
F(z)=\overline{z}A(z)+B(z),
\end{equation}
where $A$ and $B$ are analytic function (and obviously $A=F_{\overline{z}}).$

Poly-analytic functions were first introduced by the Russian mathematician G.V. Kolossov in connection with his research in the mathematical theory of elasticity who studied in particular bi-analytic mappings (See \cite{K}).
The bi-analytic functions have been introduced to study physical fields with divergence or rotation at
the present time and their theories and applications have been studied by many
authors (see \cite{S1}-\cite{FXH}). Useful applications of this idea in mechanics are widely
known from the remarkable investigations by Kolossoff, his student Muskhelishvili and
their followers. The applications of poly-analytic functions to
problems in elasticity are well documented in his book \cite{M}. Most important applications of the theory of functions of a complex variables were obtained in the plane theory of elasticity(see \cite{K}-\cite{M}). Poly-analytic
function theory has been investigated intensively, notably by the Russian
school led by Balk \cite{B}. A new characterization of poly-analytic functions has been obtained by Agranovsky \cite {A}.
 We note that a Poly-analytic complex function is poly-harmonic, but the converse is
not true. If $A_k(z)=1$ for all $k$, then $F(z)$ is a harmonic function (see \cite{B}). The
composition of poly-analytic function with conformal mapping from both sides, in
general is not poly-analytic. The properties of poly-analytic functions can be different from those enjoyed by analytic functions, for example they can vanish on closed curves without vanishing identically. One such example the function $F(z) = 1-z\overline{z}$. Still, many properties of analytic functions have found an extension to poly-analytic functions, often in a nontrivial form. 
We consider in Section 2 Landau's Theorem and we give two different versions one for poly-analytic and the other for bi-analytic functions. In section 3 we show that the Bohr's radius holds for poly-analytic functions and in section 4, we consider  the case when the poly-analytic functions have starlike analytic counterparts, we find an upper bound for the arclength of poly-analytic functions and solve the problem of minimizing moments of order $p$ for bi-analytic functions.

\section{Landau's Theorem for Poly-analytic functions}

\bigskip The classical Landau Theorem for bounded analytic functions states
that if $f$ is analytic in the unit disk $U$ with $f(0)=0,$ $f^{\shortmid
}(0)=1$ and $|f(z)|<M$ for $z\in U,$ then $f$ is univalent in the disk
$U_{\rho_{0}}=\{z:|z|<\rho_{0}\}$ with
\[
\rho_{0}=\frac{1}{M+\sqrt{M^{2}-1}}%
\]
and $f(U_{\rho_{0}})$ contains a disk $U_{R_{0}}$ with $R_{0}=M\rho_{0}^{2}.$
This result is sharp. Chen, Gauthier and Hengartner \cite{CGW} obtained a version of the Landau
theorem for bounded harmonic mappings of the unit disk. Unfortunately their result is not sharp. Better estimates were
given  later in \cite{CPW}.
In specific, it was shown in [11] that if $f$ is harmonic in the unit disk $U$
with $f(0)=0,$ $J_{f}(0)=1$ and $|f(z)|<M$ for $z\in U,$ then $f$ is univalent
in the disk $U_{\rho_{1}}=\{z:|z|<\rho_{1}\}$ with
\[
\rho_{1}=1-\frac{2\sqrt{2}M}{\sqrt{\pi+8M^{2}}}%
\]
and $f(U_{\rho_{1}})$ contains a disk $U_{R_{1}}$ with $R_{1}=\dfrac{\pi}%
{4M}-2M\dfrac{\rho_{1}^{2}}{1-\rho_{1}}.$ This result is the best known but
not sharp.
Landau's Theorem has been also proven for the classes of Biharmonic mappings
by several authors and extended to the class of Polyharmonic mappings (see \cite{AA1}- \cite{CPW}-\cite{CRW}).

We first give a version of Landau's theorem for Poly-analytic functions.

\begin{theorem}
Let $F(z)={\displaystyle\sum\limits_{k=0}^{\alpha-1}}\overline{z}^{k}A_{k}(z)$ be a poly-analytic function of $\ order\ \alpha\ $on
$U,$ where $A_{k}$ are analytic such that $A_{k}(0)=0,$ $A_{k}^{\prime}(0)=1$
and $|A_{k}|$ $\leq$ $M\ \ $for all $k.$ Then there is a constant $0<\rho
_{1}<1$ so that $F$ is univalent in $|z|<\rho_{1}.$ In specific, $\rho_{1}$
satisfies
$$1-M\left(  \frac{\rho_{1}(2-\rho_{1})}{(1-\rho_{1})^{2}}+ {\displaystyle\sum\limits_{k=1}^{\alpha-1}}\frac{\rho_{1}^{k}(1+k-\rho_{1})}{(1-\rho_{1})^{2}}\right)  =0,$$ and $F(U_{\rho_{1}})$ contains a disk $U_{R_{1}},$ where
$ R_{1}=\rho_{1}-\rho_{1}^{2}(\frac{1-\rho_{1}^{\alpha-1}}{1-\rho_{1}})-M{\displaystyle\sum\limits_{k=0}^{\alpha-1}}
\frac{\rho_{1}^{k+2}}{1-\rho_{1}}.$
\end{theorem}

\begin{proof}
 Fix $0<\rho<1$ and choose $z_{1},z_{2}$ with $z_{1}\neq z_{2}$,
$|z_{1}|<\rho$ and $|z_{2}|<\rho.$ We first note that $F_{z}={\displaystyle\sum\limits_{k=0}^{\alpha-1}}\overline{z}^{k}A_{k}^{\prime},\ \ F_{\overline{z}}(z)={\displaystyle\sum\limits_{k=0}^{\alpha-1}}k\overline{z}^{k-1}A_{k}^{\prime},$ which implies $F_{z}(0)=A_{0}^{\prime}(0)=1$, $F_{\overline{z}}(0)=A_{1}(0)=0$. We write each$\ A_{k}(z)={\displaystyle\sum\limits_{n=1}^{\infty}}
a_{n,k}z^{n}.$
 It follows that on the line segment $[z_{1},z_{2}]$ we have
\begin{align*}
&\left\vert F(z_{1})-F(z_{2})\right\vert  \\
&  ={\displaystyle\int\limits_{[z_{1},z_{2}]}}F_{z}(z)dz+F_{\overline{z}}(z)d\overline{z}\\
&  =\bigskip\left\vert{\displaystyle\int\limits_{\lbrack z_{1},z_{2}]}}(F_{z}(0)dz+F_{\overline{z}}(0)d\overline{z})+
{\displaystyle\int\limits_{[z_{1},z_{2}]}}\left(  F_{z}(z)-F_{z}(0)\right)  dz+\left(  F_{\overline{z}}(z)-F_{\overline
{z}}(0)\right)  d\overline{z})\right\vert \\
&  =\bigskip\left\vert{\displaystyle\int\limits_{\lbrack z_{1},z_{2}]}}dz+
{\displaystyle\int\limits_{[z_{1},z_{2}]}}(A_{0}^{\shortmid}(z)-A_{0}^{\prime}(0))d\bigskip z+
{\displaystyle\int\limits_{[z_{1},z_{2}]}}
{\displaystyle\sum\limits_{k=1}^{\alpha-1}}k\overline{z}^{k-1}A_{k}d\overline{z}\right\vert \\
&  \geq|z_{2}-z_{1}|-{\displaystyle\int\limits_{[z_{1},z_{2}]}}{\displaystyle\sum\limits_{n=2}^{\infty}}
n|a_{n,0}|\rho^{n-1\ }-{\displaystyle\int\limits_{[z_{1},z_{2}]}}{\displaystyle\sum\limits_{k=1}^{\alpha-1}}\left\vert \overline{z}^{k}\right\vert \left(  \left\vert A_{k}^{\prime}\right\vert +k\left\vert \frac{A_{k}}{z}\right\vert \right)  \left\vert
dz\right\vert \\
&  \geq|z_{2}-z_{1}|\left(  1-{\displaystyle\sum\limits_{n=1}^{\infty}}(n+1)|a_{n,0}|\rho^{n\ }-{\displaystyle\sum\limits_{k=1}^{\alpha-1}}\rho^{k}{\displaystyle\sum\limits_{n=1}^{\infty}}\left(  n|a_{n,k}|+ka_{n,k}\right)  \rho^{n-1}\right)  \\
&  \geq|z_{2}-z_{1}|\left(  1-M\left(  \frac{\rho}{(1-\rho)^{2}}+\frac{\rho}{1-\rho}\right)  -M{\displaystyle\sum\limits_{k=1}^{\alpha-1}}\rho^{k}(\frac{1}{(1-\rho)^{2}}+\frac{k}{1-\rho})\right)  \\
&  =|z_{2}-z_{1}|\left(  1-M\left(  \frac{\rho(2-\rho)}{(1-\rho)^{2}}+{\displaystyle\sum\limits_{k=1}^{\alpha-1}}\frac{\rho^{k}(1+k-\rho)}{(1-\rho)^{2}}\right)  \right).
\end{align*}

Clearly there is a $\rho$ so that $\left\vert F(z_{1})-F(z_{2})\right\vert
>0.$ Let $\rho_{1}$ be the largest such $\rho.$ In other words, choose
$\rho_{1}>0$ so that $$1-M\left(  \frac{\rho_{1}(2-\rho_{1})}{(1-\rho_{1})^{2}}+{\displaystyle\sum\limits_{k=1}^{\alpha-1}}
\frac{\rho_{1}^{k}(1+k-\rho_{1})}{(1-\rho_{1})^{2}}\right)  =0.$$

For $|z|=\rho_{1},$%
\begin{align*}
|F(z)|  &  =\left\vert
{\displaystyle\sum\limits_{k=0}^{\alpha-1}}
\overline{z}^{k}A_{k}(z)\right\vert =\left\vert a_{1,0}\ z+%
{\displaystyle\sum\limits_{k=1}^{\alpha-1}}
a_{1,k}|z|^{2}\overline{z}^{k-1}+%
{\displaystyle\sum\limits_{k=0}^{\alpha-1}}
\overline{z}^{k}%
{\displaystyle\sum\limits_{2}^{\infty}}
a_{n,k}z^{n\ }\right\vert \\
&  \geq\rho_{1}-%
{\displaystyle\sum\limits_{k=2}^{\alpha-1}}
\rho_{1}^{k}-M%
{\displaystyle\sum\limits_{k=0}^{\alpha-1}}
{\displaystyle\sum\limits_{2}^{\infty}}
\rho_{1}^{n+k}\\
&  =\rho_{1}-\rho_{1}^{2}(\frac{1-\rho_{1}^{\alpha-1}}{1-\rho_{1}})-M%
{\displaystyle\sum\limits_{k=0}^{\alpha-1}}
\frac{\rho_{1}^{k+2}}{1-\rho_{1}}.
\end{align*}

\end{proof}

As a corollary, we conclude a version of Landau's theorem for bi-analytic functions :
\begin{corollary}
Let $F(z)=\overline{z}A(z)+B(z)$ be a Bi-analytic function of $U,$ where $A$
and $B$ are analytic such that $A(0)=B(0)=0,$ $A^{\prime}(0)=B^{\prime}(0)=1$
and $|A|$ and $|B|$ are both bounded by $M.$ Then there is a constant
$0<\rho_{1}<1$ so that $F$ is univalent in $|z|<\rho_{1}.$ In specific,
$\rho_{1}$ satisfies%
\[
1-2M\left(  \frac{2\rho_{1}-\rho_{1}^{2}}{(1-\rho_{1})^{2}}\right)  =0,
\]
that is $$\rho_{1}=\frac{2M}{2M+1}\left(  1+\sqrt{\frac{2M+1}{2M}}+\frac{1}%
{2M}\right).  $$ Moreover, $F(U_{\rho_{1}})$ contains a disk $U_{R_{1}%
},$ where%
\[
R_{1}=\rho_{1}-\rho_{1}^{2}-M\frac{\rho_{1}^{3}+\rho_{1}^{2}}{1-\rho_{1}}.
\]

\end{corollary}

\section{Bohr's theorem for Poly-analytic functions}

Bohr's inequality says that if $f(z)=\underset{n=0}{\overset{\infty}{%
{\displaystyle\sum}
}}a_{n}z^{n}\ $ is analytic in the unit disc $U$ and $|f(z)|<1$ for all $z$ in
$U,$ then $\underset{n=0}{\overset{\infty}{%
{\displaystyle\sum}
}}\left\vert a_{n}z^{n}\right\vert \leq1\ $for all $z\in$ $U$ with
$|z|\leq\frac{1}{3}.$ This inequality was discovered by Bohr in 1914 (see\cite{Bo}).
Bohr actually obtained the inequality for $|z|\leq\frac{1}{6}.$ Later Wiener,
Riesz and Schur independently, established the inequality for $|z|\leq\frac{1}{3}$ and showed that $\frac{1}{3}$ is sharp (See \cite{PPS},{\cite{S2},\cite{T}). More recently, the Bohr's inequality has been of interest for several authors. We refer to \cite{AAP} for a historical overview.

Let $A(z)=\underset{n=0}{\overset{\infty}{\sum}}a_{n}z^{n}$ be analytic
function in $U$ and let $M(A):=\underset{n=0}{\overset{\infty}{\sum}}%
|a_{n}|r^{n}$ be the associated majorant series for $A$ . If $A$ and $B$ are
analytic functions, straight forward calculations we deduce the following:

\begin{center}%
\begin{equation}
M(A+B)\leq M(A)+M(B)\label{eq1}%
\end{equation}

\begin{equation}
M(AB)\leq M(A)M(B).\label{eq2}%
\end{equation}

\end{center}

In the next result, we establish a Bohr's type inequality for
Poly-analytic functions of order $\alpha$ given by $F(z)={\displaystyle\sum\limits_{k=0}^{\alpha-1}}\overline{z}^{k}A_{k}(z)$. For this family of functions, we define the majorant series $M(F,r) $   by 
$$M(F,r)=\underset{n=0}{\overset{\infty}{{\displaystyle\sum}}}{\displaystyle\sum\limits_{k=0}^{\alpha-1}}|a_{n,k}|r^{k+n}.$$
\begin{theorem}
Let $F(z)={\displaystyle\sum\limits_{k=0}^{\alpha-1}}\overline{z}^{k}A_{k}(z)$ be a Poly-analytic function of order $\alpha,$ where $A_k$ are analytic mappings for $k=0,1,...\alpha-1$, such
that  $f_{k}(z)=A_{0}(z)+\overline{A_{k}(z)}$ preserves orientation for each $k$. Suppose that $A_{0}\ $is\ univalent and normalized by $A_{0}(0)=0,$
$A_{0}^{\prime}(0)=1$ and \ $\ F(U)\subset U.$ Then
\[
\ M(F,r)<1\ \ \ \ \text{if\ \ \ \ }|z|<r_0,
\]
where $r_0$ is the root of the polynomial $r^{\alpha}+r^{\alpha-1}+...+r^3+3r-1=0$. We can take $r_0 \approx 0.318.$
\end{theorem}

\begin{proof}
We first note that
\[
M(F,r)=M\left(
{\displaystyle\sum\limits_{k=0}^{\alpha-1}}
\overline{z}^{k}A_{k}(z)\right)  \leq%
{\displaystyle\sum\limits_{k=0}^{\alpha-1}}
r^{k}M(A_{k}).
\]

Now since $f_{k}(z)=A_{0}(z)+\overline{A_{k}(z)}$ preserves the orientation
for each $k,$ it follows that $A_{k}^{\prime}(z)=a_{k}(z)A_{0}^{\prime}(z)$
for $a_{k}\in H(U)$ and $|a_{k}(z)|<1$ for all $z\in U$, which gives that
$$M(A_{k})\leq{\displaystyle\int\limits_{0}^{r}}M(A_{k}^{\prime})ds={\displaystyle\int\limits_{0}^{r}}M(a_{k}A_{0}^{\prime})ds\leq
{\displaystyle\int\limits_{0}^{r}}M(A_{0}^{\prime})ds=M(A_{0}).
$$

Hence
\begin{equation}\label{inequalityBohr}
M(F,r)\leq{\displaystyle\sum\limits_{k=0}^{\alpha-1}}r^{k}M(A_{0})=M(A_{0})\frac{1-r^{\alpha}}{1-r}.
 \end{equation}
We next find a bound for $M(A_{0}).$We have %
$$
M(A_{0})={\displaystyle\int\limits_{0}^{r}} M(A_{0}^{\prime})ds \leq
{\displaystyle\int\limits_{0}^{r}}\underset{n=0}{\overset{\infty}{{\displaystyle\sum}}}\left\vert na_{n,0}z^{n-1}\right\vert ds$$
$$={\displaystyle\int\limits_{0}^{r}}\underset{n=0}{\overset{\infty}{{\displaystyle\sum}}}n|a_{n,0}|s^{n-1}ds
\leq {\displaystyle\int\limits_{0}^{r}}\underset{n=0}{\overset{\infty}{{\displaystyle\sum}}}nr^{n}=\dfrac{r}{(1-r)^{2}}.
$$

Therefore
\[
M(F,r)\leq\frac{r(1-r^{\alpha})}{(1-r)^{3}}.
\]

It follows that for all integer values of $\alpha$, we have $M(F,r)\leq
\frac{r(1-r^{\alpha})}{(1-r)^{3}}<1$ for $|z|<r_0$, where $r_0$ is the root of the polynomial $r^{\alpha}+r^{\alpha-1}+...+r^3+3r-1=0$. 

Below is a table that indicates the value of the Bohr's radius depending on
the value of $\alpha.$ This tables shows that that we can take $r_0 \approx 0.318.$

\[%
\begin{tabular}
[c]{|l|l|}%
$\alpha$ & Bohr's Radius\\
\hline
2 & $1/3$\\
3 & $0.322$\\
4 & $0.319$\\
5 & 0.318\\
50 & 0.318\\
100 & 0.318
\end{tabular}
\]

\end{proof}

As a corollary, and as seen from the above table we obtain a Bohr's type inequality for Bi-analytic functions. We note here that if
$F(z)=\overline{z}A(z)+B(z)$ is a Bi-analytic function we take %

\[
M(F,r)=\underset{n=0}{\overset{\infty}{%
{\displaystyle\sum}
}}\left\vert a_{n}\right\vert r+|b_{n}|.%
\]
\begin{corollary}
Let $F(z)=\overline{z}A(z)+B(z)$ be a Bi-analytic function such that
$f(z)=\overline{A}(z)+B(z)$ preserves the orientation. Suppose that
$B $ is univalent and normalized by $B(0)=0,$ $B^{\prime}(0)=1$ and
$ F(U)\subset U.$ Then
\[
\ M(F,r)<1\ \ \ \ \text{if\ \ \ \ }|z|<\frac{1}{3}.
\]
The bound is sharp and is attained by suitable rotation of the Koebe function
$A(z)=\dfrac{z}{(1-z)^{2}}.$
\end{corollary}

Recently, Abu Muhanna has shown the following lemma for analytic functions \cite{AM}, which gives the radius under which the majorant function is bounded by the distance to the boundary of the image of an analytic function.
\begin{lemma}\label{abumuhanna}
 Let $A(z)=\sum\limits_{0}^{\infty }a_{n}z^{n}$ be analytic on $U$
and suppose that $A(U)$ misses at least \ two points then 
$$M(A)\leq dist(A(0),\partial A(U)),$$
for $|z|\leq e^{-\pi }= 4.321\,4\times 10^{-2}.$
\end{lemma}

We generalize the above lemma to poly-analytic functions under certain conditions on $A_k.$
\begin{theorem}
Let $F(z)={\displaystyle\sum\limits_{k=0}^{\alpha-1}}\overline{z}^{k}A_{k}(z)$ be a Poly-analytic function of order $\alpha,$ where $A_k$ are analytic mappings for $k=0,1,...\alpha-1$,  such that  $f_{k}(z)=A_{0}(z)+\overline{A_{k}(z)}$ preserves orientation for each $k$. Suppose that $A_0(z)$ misses at least two points and $A_0(0)=a_{0},$
$A_k(0)=0,\,\,\,$ for $k=1,...,\alpha-1$. Then if $\ |z|<e^{-\pi},$ %
\[
M(F,r)\leq\frac{1-r^{\alpha}}{1-r}d(a_{0},\partial A_0(U)).
\]
\end{theorem}
\begin{proof}
The proof is is a simple application of equation (\ref{inequalityBohr}) and Lemma \ref{abumuhanna}.
\end{proof}
\begin{corollary}
Let $F(z)=\overline{z}A(z)+B(z)$ be a Bi-analytic function such that
$f(z)=B(z)+\overline{A(z)}$ preserves the orientation, where $B(0)=b_{0},$
$A(0)=0$, and suppose that $B(z)$ misses at least two points. Then if $\ |z|<e^{-\pi},$ we have%
\[
M(F,r)\leq(1+r)d(b_{0},\partial B(U)).
\]

\end{corollary}

\section{Poly-analytic functions with starlike counterparts}

We first recall that an analytic univalent mapping $A(z)$ is said to be starlike analytic if it satisfies 
\begin{equation*}
\dfrac{\partial \arg A(re^{i\theta })}{\partial \theta }=\Re\frac{zA'(z)}{A(z)}>0
\end{equation*}%
for all $z\in U.$

The next result establishes an upper estimate for the arclength of Poly-analytic functions with starlike analytic counterparts.

\begin{theorem}
Let $F(z)=%
{\displaystyle\sum\limits_{k=1}^{\alpha-1}}
\overline{z}^{k}A_{k}(z)\ $be a Poly-analytic functions such that each
$A_{k}(z)$ is a starlike analytic function and
$|A_{k}(z)|\leq M(r)$,for each $k=1,...\alpha-1,$ \,\,$0<r<1$. Let \ $L(r)$ \ denote the arclength of the
curve $C_{r},\ $where $C_{r}$ denote the image of the circle $|z|=r<1$ under
the function $w=F(z).$ Then
\[
L(r)\leq \dfrac{2\pi M(r)r}{1-r}\left[ \dfrac{(1+r)\left( (\alpha-1)r^{\alpha}-\alpha r^{\alpha-1}+1\right)}{(1-r)^2}-1+r^{\alpha-1}\right].
\]

\begin{proof}
We have 
\begin{align*}
L(r)  &  =\int_{C_{r}}|dF|\ =\int_{0}^{2\pi}|zF_{z}-\overline{z}%
F_{\overline{z}}|d\theta \\
&  =\int_{0}^{2\pi}\left\vert{\displaystyle\sum\limits_{k=1}^{\alpha-1}}\overline{z}^{k}zA_{k}^{\prime}(z)-\overline{z}k\overline{z}^{k-1}A_{k}(z)\right\vert d\theta\\
&  \leq \int_{0}^{2\pi}{\displaystyle\sum\limits_{k=1}^{\alpha-1}}r^{k}|A_{k}(z)|\left\vert \dfrac{zA_{k}^{\prime}(z)}{A_{k}(z)}-k\right\vert
d\theta\\
&  \leq M(r)\int_{0}^{2\pi}{\displaystyle\sum\limits_{k=1}^{\alpha-1}}r^{k}\int_{0}^{2\pi}\left\vert \dfrac{zA_{k}^{\prime}(z)}{A_{k}(z)}-k\right\vert d\theta.
\end{align*}
Since $A_{k}(z)$ is a starlike analytic function, we have
$\operatorname{Re}\left(  \dfrac{zA_{k}^{\prime}(z)}{A_{k}(z)}\right)
>0\ $\ and it follows that $\dfrac{zA_{k}^{\prime}(z)}{A_{k}(z)}-k$ is
subordinate to $ \dfrac{1+z}{1-z}-k= (1-k)\dfrac{1}{1-z}+(1+k)\dfrac{z}{1-z}.$ 
Therefore,

$$L(r)  \leq  M(r){\displaystyle\sum\limits_{k=1}^{\alpha-1}}r^{k}\int_{0}^{2\pi}(1-k)\left\vert\dfrac{1}{1-z}\right\vert+(1+k)\left\vert \dfrac{z}{1-z}\right\vert d\theta.$$
We note that $$\left\vert \frac{1}{1-z}\right\vert =\left\vert  \underset{n=0}{\overset{\infty}{\sum}}z^{n}\right\vert \leq  \left\vert \underset{n=0}{\overset{\infty}{\sum}}r^{n}\right \vert=\frac{1}{1-r},$$
and
$$\left\vert  \frac{1+z}{1-z}\right\vert \leq  1+2\underset{n=1}{\overset{\infty}{\sum}}r^{n}
=\frac{1+r}{1-r}.$$
Hence,

\begin{align*}
L(r) & \leq \frac{2\pi M(r)}{1-r}\displaystyle\sum\limits_{k=1}^{\alpha-1}{r^{k}\left(k-1+r(k+1)\right)}\\
&\leq \frac{2\pi M(r)}{1-r}\left[(1+r)\displaystyle\sum\limits_{k=1}^{\alpha-1}{kr^{k}}-(1-r)\displaystyle\sum\limits_{k=1}^{\alpha-1}{r^{k}}\right]\\
&\leq \frac{2\pi r M(r)}{1-r} \left[\dfrac{(1+r)\left( (\alpha-1)r^{\alpha}-\alpha r^{\alpha-1}+1\right)}{(1-r)^2}-1+r^{\alpha-1}\right].
\end{align*}

\end{proof}
\end{theorem}

In the special case of bi-analytic functions, we let $\alpha=2$ and obtain by a simple calculation the following corollary:

\begin{corollary}
Let $F(z)=\overline{z}A(z)\ $be a Bi-analytic functions such that $A(z)$ is a
starlike analytic function. Suppose that $|A(z)|\leq M(r)$, $0<r<1$. Suppose
that $|A_{k}(z)|\leq M(r)$, $0<r<1$. Let \ $L(r)$ \ denote the arclength of
the curve $C_{r},\ $where $C_{r}$ denote the image of the circle $|z|=r<1$
under the function $w=F(z).$ Then%
\[
L(r)\leq4\pi M(r)\dfrac{r^{2}}{1-r}.
\]

The function $F(z)=\dfrac{\overline{z}z}{(1-z)^{2}}$ shows the result is best possible.
\end{corollary}

We next consider the problem of minimizing the moments of order $p$ over a
subclass of the class bi-analytic functions defined over the unit disc $U$. 

\begin{theorem}\label{moment}
Let $f(z)=\overline{z}A(z)$ be a bi-analytic function defined on the unit
disk $U$ such that  $A(z)/z$ is starlike univalent analytic. Let 
$M_{p}(r,f)$ denote the moment of order $p$, $p\geq 0.$ Then,

$$M_{p}(r,f)\geq 2\pi \dfrac{r^{3p+6}}{3p+6}\,\,.$$

Equality holds if and only if $f(z)=\overline{z}z^{2}$
\end{theorem}

\begin{proof}
Let $f(z)=\overline{z}A(z)$ be a bi-analytic function defined on the unit
disk $U$ where $A(z)=z\phi (z),$ $\phi $ is starlike univalent analytic. Let 
$M_{p}(r,f)$ denote the moment of order $p$, $p\geq 0.$ Then,

$M_{p}(r,f)=\int_{0}^{r}\int_{0}^{2\pi }|f|^{p}\left( \left\vert
f_{z}\right\vert ^{2}-\left\vert f_{\overline{z}}\right\vert ^{2}\right)
\rho d\theta d\rho $

\[
=\int_{0}^{r}\int_{0}^{2\pi }\rho ^{3p}\left\vert \frac{\phi (z)}{z}%
\right\vert ^{p}\left( \left\vert \overline{z}\phi (z)+|z|^{2}\phi ^{\prime
}(z)\right\vert ^{2}-\left\vert z\phi (z)\right\vert ^{2}\right) \rho
d\theta d\rho 
\]

\begin{equation}
=\int_{0}^{r}\int_{0}^{2\pi }\rho ^{3p}\left\vert \frac{\phi (z)}{z}
\right\vert ^{p}\left( \rho ^{4}\left\vert \phi ^{\prime }(z)\right\vert
^{2}+2\rho ^{2}\left\vert \phi (z)\right\vert ^{2}\Re \dfrac{z\phi ^{\prime
}(z)}{\phi (z)}\right) \rho d\theta d\rho.
\end{equation}

Since $\phi (z)$ is starlike, it follows that $\Re \dfrac{z\phi ^{\prime }(z)%
}{\phi (z)}>0.$

Hence,

\begin{equation}
M_{p}(r,f)\geq \int_{0}^{r}\int_{0}^{2\pi }\rho ^{3p+5}\left\vert \frac{%
\phi (z)}{z}\right\vert ^{p}\left\vert \phi ^{\prime }(z)\right\vert^{2} d\theta
d\rho.
\end{equation}

Writing 

$$\left( \frac{\phi (z)}{z}\right) ^{p/2}\phi ^{\prime }(z)=1+\overset{\infty}{\underset{k=1}{\sum }}c_{k}z^{k},$$

we have 

$$\int_{0}^{2\pi }\rho ^{3p+5}\left\vert \frac{\phi (z)}{z}\right\vert
^{p}\left\vert \phi ^{\prime }(z)\right\vert^{2} d\theta =2\pi \rho
^{3p+5}\left( 1+\overset{\infty }{\underset{k=1}{\sum }}|c_{k}|^2|z^{k}|^2\right) .$$

Therefore,

$$ M_{p}(r,f)\geq 2\pi \int_{0}^{r}\rho ^{3p+5}d\rho =2\pi \dfrac{r^{3p+6}}{3p+6}\,\, .$$
Equality holds if and only if 
$\left( \frac{\phi (z)}{z}\right) ^{p/2}\phi ^{\prime }(z)=1$ which gives $\phi (z)= z $ and then $f(z)=\overline{z}z^{2}.$
\end{proof}
\begin{remark}
If $p=0$ in Theorem \ref{moment}, then we have the problem of minimizing the area.
Moreover, if $p=2,$ then we obtain the minimum of the moment of inertia. 
\end{remark}

We can in a similar way obtain the minimal area for a more general bi-analytic function.
\begin{theorem}
Let $f(z)=\overline{z}A(z)+B(z)$ be a bi-analytic function defined on the unit
disk $U$ such that  $A(z)/z$ is starlike univalent analytic and $Re(\overline{z}A'\overline{B'})\geq 0$. Then, the minimal area is given by

$$\underset{U_r}{\int\int} J_F dA\geq \frac{\pi r^{6}}{3},$$
where $U_r = \{z:|z|\leq r\}$.
\end{theorem}

\begin{proof}
Proceeding as in the proof of the previous theorem we get
\begin{align*}
\underset{U_R}{\int\int}  J_F dA &= \int_{0}^{r}\int_{0}^{2\pi }\left( \left\vert f_{z}\right\vert ^{2}-\left\vert f_{\overline{z}}\right\vert ^{2}\right)\rho d\theta d\rho \\
&=\int_{0}^{r}\int_{0}^{2\pi }\left(|\overline{z}A'|^{2} +|B'|^2+2\Re\overline{z}A'\overline{B'}-|A|^2\right) \rho
d\theta d\rho \\
&\geq \int_{0}^{r}\int_{0}^{2\pi }\left( \rho ^{4}\left\vert \phi ^{\prime }(z)\right\vert^{2}+2\rho ^{2}\left\vert \phi (z)\right\vert ^{2}\Re \dfrac{z\phi ^{\prime}(z)}{\phi (z)}\right) \rho d\theta d\rho \\
&\geq \int_{0}^{r}\int_{0}^{2\pi } \rho ^{5}\left\vert \phi ^{\prime }(z)\right\vert^{2} d\theta d\rho, 
\end{align*}

since $\Re \dfrac{z\phi ^{\prime}(z)}{\phi (z)}\>0$. Now, $\phi'(z) = 1+\overset{\infty}{\underset{k=1}{\sum }}c_{k}z^{k} $, hence  $\int_{0}^{2\pi } |\phi'(z)|^2 d\theta \geq 1 .$
It follows that $$\underset{U_r}{\int\int}  dA\geq \frac{\pi r^{6}}{3} \,\, .$$
\end{proof}
Finally in the next theorem, we establish a linkage between starlike analytic
functions and Bi-analytic functions.

\begin{theorem}
Let $F(z)=\overline{z}A(z)$ be a Bi-analytic function, where $A\in H(U),$
with $A(0)=0$ and $A^{\prime}(0)=1.$  $A$ is starlike if and only if
$\Phi(z)=zF(z)$ is starlike.
\end{theorem}
\begin{proof}
 Suppose $A$ is starlike, then by direct calculations we have 

\begin{align*}
J_{\Phi}  & =\left\vert \Phi_{z}\right\vert ^{2}-\left\vert \Phi_{\overline
{z}}\right\vert ^{2}=\left\vert F+zF_{z}\right\vert ^{2}-\left\vert
zF_{\overline{z}}\right\vert ^{2}\\
& =|F|^{2}+|zF_{z}|^{2}+2\Re zF_{z}\overline{F}-\left\vert zF_{\overline{z}%
}\right\vert ^{2}\\
& =|zF_{z}|^{2}+2|F|^{2}\Re\frac{zF_{z}}{F}\\
&=|zA'|^{2}+2|F|^{2}\Re\frac{zA'}{A}\\
& >0
\end{align*}

\ if$\ z\neq 0.\ \ \ \ $ Moreover,%

\[
\Re\frac{z\Phi_{z}-\overline{z}\Phi_{\overline{z}}}{\Phi}=\Re\left(
1+\frac{zF_{z}}{F}-\frac{\overline{z}F_{\overline{z}}}{F}\right) =\Re\left(
\frac{zF_{z}}{F}\right) =  \Re\left(\frac{zA'}{A}\right) >0,
\]

since $\frac{\overline{z}F_{\overline{z}}}{F} = 1$. The converse follows in a similar fashion.


\end{proof}

\bigskip

\bigskip
\end{document}